\documentclass[11pt]{amsart}
\usepackage[utf8]{inputenc}
\usepackage[T1]{fontenc}
\usepackage[english]{babel}

\usepackage[a4paper, margin=3.5cm]{geometry}

\usepackage{tikz}

\usepackage{amssymb}
\usepackage{empheq}
\usepackage{esint}
\usepackage{xcolor}
\usepackage{hyperref}
\hypersetup{colorlinks = true, linkbordercolor = red, citecolor = blue}

\usepackage[alphabetic, msc-links]{amsrefs}
\renewcommand{\eprint}[1]{\href{https://arxiv.org/abs/#1}{\texttt{arXiv:#1}}}

\numberwithin{equation}{section}

\newtheoremstyle{newthm}
{.5em} 
{.5em} 
{\it}       
{}          
{\bf}       
{.}         
{.5em}      
{}          %
\theoremstyle{newthm}
\newtheorem{thm}{Theorem}[section]

\newtheorem{prp}[thm]{Proposition}
\newtheorem{conj}[thm]{Conjecture}

\newtheoremstyle{newdef}
{.5em} 
{.5em} 
{}          
{}          
{\bf}       
{.}         
{.5em}      
{}          %
\theoremstyle{newdef}
\newtheorem{defn}[thm]{Definition}
\newtheorem{rem}[thm]{Remark}

\newcommand{\R}{\mathbf{R}}

\renewcommand{\S}{\mathbf{S}}

\renewcommand{\a}{\alpha}

\newcommand{\g}{\gamma}
\renewcommand{\d}{\delta}
\newcommand{\e}{\varepsilon}

\renewcommand{\th}{\theta}
\renewcommand{\k}{\kappa}

\newcommand{\n}{\nu}
\newcommand{\x}{\xi}

\renewcommand{\t}{\tau}

\newcommand{\G}{\Gamma}

\newcommand{\p}{\partial}

\begin{document}

\title{Singularities of the network flow with symmetric initial data}

\author{Matteo Novaga}
\address{Dipartimento di Matematica, Università di Pisa, Largo Bruno Pontecorvo 5, 56127 Pisa, Italy.}
\email{matteo.novaga@unipi.it}

\author{Luciano Sciaraffia}
\address{Dipartimento di Matematica, Università di Pisa, Largo Bruno Pontecorvo 5, 56127 Pisa, Italy.}
\email{luciano.sciaraffia@phd.unipi.it}

\subjclass[2010]{}

\begin{abstract}
We study the formation of singularities for the curvature flow of networks when the initial data is symmetric with respect to a pair of perpendicular axes and has two triple junctions.
We show that, in this case, the set of singular times is finite.
\end{abstract}

\maketitle

\tableofcontents


\section{Introduction}\label{sec:intro}

The \textit{Mean Curvature Flow} is one of the best studied geometric evolution equations, in particular its one-dimensional version, often called the \textit{Curve Shortening Flow}.
This last flow is completely understood thanks to the works of Gage--Hamilton and Grayson \cites{GageHamilton-Convex1986, Grayson-RoundPoints1987}: a closed, embedded curve in the plane becomes convex in finite time and then shrinks to a \textit{round point}.
A natural and interesting generalisation of this flow is the \textit{Network Flow}, also known as \textit{Multiphase Mean Curvature Flow} in higher dimensions, where instead of considering a single curve the underlying geometric object is a \textit{regular network}, that is, a finite union of embedded curves which can meet only at their endpoints, and at each multiple junction only three curves meet forming equal angles of $2\pi/3$ 
(more precise definitions will be provided in Section \ref{sec:setup}).
This last condition, called after Herring, arises naturally because of the variational structure of the flow, since these triple junctions minimise length locally.

The network flow has been thoroughly studied, although a complete understanding as in the case of a single curve is far from being achieved.
One of the first results in this line comes from Bronsard--Reitich \cite{Bronsard-Reitich-1993}, where they showed short time existence for the flow of \textit{triods}, i.e. networks consisting of three curves and one triple junction, and with Neumann boundary conditions.
Subsequently, the works of Mantegazza--Novaga--Tortorelli \cite{MNT-NetsI-2004} and later Magni--Mantegazza--Novaga \cite{MMN-NetsII-2016} studied the singularity formation under the flow, with Dirichlet boundary conditions, stating in which cases the flow exists for all times and reaches in the limit the Steiner tree spanned by the three endpoints.

More recently, in \cites{Goesswein-Menzel-Pluda-Existence2023, MNPS-Survey2018} a general proof of existence of a solution to the network flow with regular initial data was given.
It was also shown that the flow can be extended to a maximal existence time at which, if finite, a singularity forms: either the $L^2$-norm of the curvature blows-up, or the length of one of the curves goes to zero.
Moreover, as in the case of the curve shortening flow, there holds a \textit{geometric uniqueness:} every other solution starting at the same initial network is just a reparametrisation of the flow.
Thus, to give a complete description of the flow, it becomes crucial to understand and classify the singularities that can arise.

In contrast to what happens in the case of a single curve, it could be possible during the network flow that the length of one or more curves goes to zero while the curvature of the network remains bounded.
This kind of phenomenon is often called a \textit{type-0 singularity}, and allows the flow to approach an irregular network with junctions of multiplicity greater than three.
Because of this, it becomes a compelling question to understand if it is possible to start a regular flow when the initial network fails to satisfy the Herring condition.
It turns out that, thanks to the results of \cites{Ilmanen-Neves-Schulze-Shortexistence-2019, LMPS-2021}, such an irregular network can serve as initial data for a regular flow, enabling is continuation beyond some singularities, albeit not in a unique way.
Such flows are constructed by locally replacing an irregular junction by one of the self-similar, tree-like expanding solitons obtained in \cite{Mazzeo-Saez-2011} according to the number of curves concurring at said junction.
In this way, new edges might “emerge” flowing out of the junction, and the nonuniqueness of the continuation is directly tied to the nonuniqueness of the expanders.
In any case, the number of possible geometric solutions is classified by the number of expanders at each irregular junction \cite{LMPS-2021}*{Cor. 8.7}.
Thanks to these findings, it becomes possible to continue the flow past a type-0 singularity.
We remark that it could also be possible to restart the flow in other potential scenarios where the curvature does blow-up, but these are not within the scope of this discussion.
The interested reader might see the discussion present in \cite{MNPS-Survey2018}*{Sec. 10.4}.

Much of the analysis of singularities can be done conditional to the so-called \textit{multiplicity-one conjecture}, which states that every limit of parabolic rescalings of the flow around a fixed point is a flow of embedded networks with multiplicity one.
Indeed, Mantegazza--Novaga--Pluda \cite{MNP-Type0trees2022} showed, conditionally to this statement, that if there are no loops in the initial network, or in other words, the initial datum is a tree, then only type-0 singularities can occur.
Hence, in the case of trees the flow could in principle be continued indefinitely.

It turns out that the multiplicity-one conjecture is true for networks with at most two triple junctions, as it was shown in \cite{MNP-2juncs-2017}.
As a result, a complete description of the possible singularities was obtained in this case.
In particular, if no loop disappears at a singular time, there is only one sensible way to continue the flow, which the authors call the \textit{standard transition} (cf. \cite{LMPS-2021}*{Cor. 8.8}).
This situation arises when the two triple junctions coalesce into a single point, forming a quadruple junction with equal opposing angles of $\pi/3$ and $2\pi/3$.
However, the previous results only give a short time existence with no uniform control over the lifespan of the flow, which makes it difficult to rule out the possible accumulation of singular times.
This is the only question remaining to be answered to give a complete description and a global time existence theorem in this case.

In this note we address this problem in the case of networks with two triple junctions which are symmetric to a pair of perpendicular axes.
We will refer to this class of networks simply as \textit{symmetric}.
With this condition, there are only four possible cases: the \textit{tree}, the \textit{lens}, the $\th$-\textit{network}, and the \textit{eyeglasses}, which are illustrated in Figure \ref{fig:net-types}.

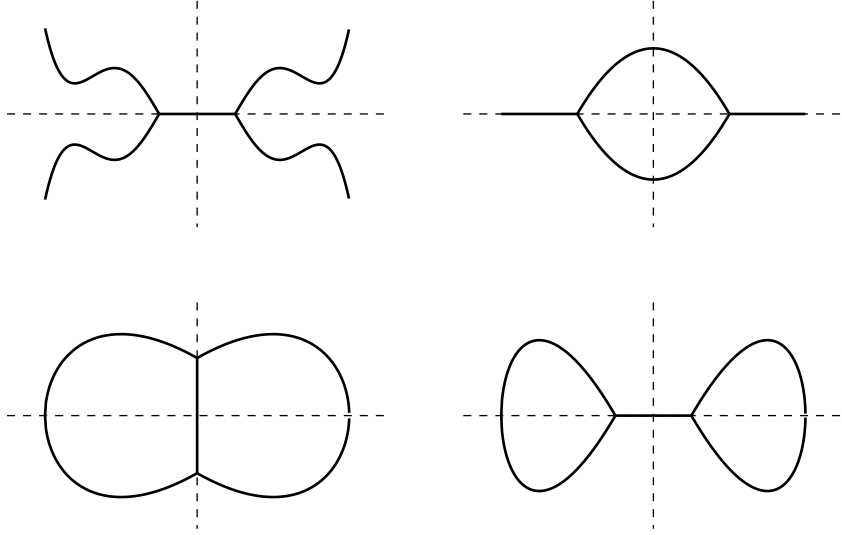
\begin{figure}
\centering

\begin{tikzpicture}[x=1cm,y=1cm,scale=1]

\draw[line width=.5pt, dashed, shift={(-3 cm, 4 cm)}] (-2.5,0) -- (2.5,0);
\draw[line width=.5pt, dashed, shift={(-3 cm, 4 cm)}] (0,1.5) -- (0,-1.5);

\draw[line width=1pt, smooth, shift={(-3 cm, 4 cm)}] plot[samples=200,domain=.5:2] (\x, {1.73*(\x-.5)*((\x-.5)^3-1.75*(\x-.5)^2+1)});
\draw[line width=1pt, smooth, shift={(-3 cm, 4 cm)}] plot[samples=200,domain=.5:2] (\x, {-1.73*(\x-.5)*((\x-.5)^3-1.75*(\x-.5)^2+1)});
\draw[line width=1pt, smooth, shift={(-3 cm, 4 cm)}] (-.5,0) -- (.5,0);
\draw[line width=1pt, smooth, shift={(-3 cm, 4 cm)}] plot[samples=200,domain=-2:-.5] (\x,{1.73*(-\x-.5)*((-\x-.5)^3-1.75*(-\x-.5)^2+1)});
\draw[line width=1pt, smooth, shift={(-3 cm, 4 cm)}] plot[samples=200,domain=-2:-.5] (\x,{-1.73*(-\x-.5)*((-\x-.5)^3-1.75*(-\x-.5)^2+1)});

\draw[line width=.5pt, dashed, shift={(3 cm, 4 cm)}] (-2.5,0) -- (2.5,0);
\draw[line width=.5pt, dashed, shift={(3 cm, 4 cm)}] (0,1.5) -- (0,-1.5);

\draw[line width=1pt, smooth, shift={(3 cm, 4 cm)}] plot[samples=200,domain=-1:1] (\x,{.87*(\x*\x-1)});
\draw[line width=1pt, smooth, shift={(3 cm, 4 cm)}] plot[samples=200,domain=-1:1] (\x,{-.87*(\x*\x-1)});
\draw[line width=1pt, smooth, shift={(3 cm, 4 cm)}] (-2,0) -- (-1,0);
\draw[line width=1pt, smooth, shift={(3 cm, 4 cm)}] (1,0) -- (2,0);

\draw[line width=.5pt, dashed, shift={(-3 cm, 0 cm)}] (-2.5,0) -- (2.5,0);
\draw[line width=.5pt, dashed, shift={(-3 cm, 0 cm)}] (0,1.5) -- (0,-1.5);

\draw[line width=1pt, smooth, shift={(-3 cm, 0 cm)}] plot[samples=2000,domain=0:2] (\x,{.54*(\x+1)*sqrt(2-\x)});
\draw[line width=1pt, smooth, shift={(-3 cm, 0 cm)}] plot[samples=2000,domain=0:2] (\x,{-.54*(\x+1)*sqrt(2-\x)});
\draw[line width=1pt, smooth, shift={(-3 cm, 0 cm)}] (0,.77) -- (0,-.77);
\draw[line width=1pt, smooth, shift={(-3 cm, 0 cm)}] plot[samples=200,domain=-2:0] (\x,{.54*(-\x+1)*sqrt(2+\x)});
\draw[line width=1pt, smooth, shift={(-3 cm, 0 cm)}] plot[samples=200,domain=-2:0] (\x,{-.54*(-\x+1)*sqrt(2+\x)});

\draw[line width=.5pt, dashed, shift={(3 cm, 0 cm)}] (-2.5,0) -- (2.5,0);
\draw[line width=.5pt, dashed, shift={(3 cm, 0 cm)}] (0,1.5) -- (0,-1.5);

\draw[line width=1pt, smooth, shift={(3 cm, 0 cm)}] plot[samples=2000,domain=.5:2] (\x,{(\x-.5)*sqrt(4-2*\x)});
\draw[line width=1pt, smooth, shift={(3 cm, 0 cm)}] plot[samples=2000,domain=.5:2] (\x,{-(\x-.5)*sqrt(4-2*\x)});
\draw[line width=1pt, smooth, shift={(3 cm, 0 cm)}] (-.5,0) -- (.5,0);
\draw[line width=1pt, smooth, shift={(3 cm, 0 cm)}] plot[samples=200,domain=-2:-.5] (\x,{(-\x-.5)*sqrt(4+2*\x)});
\draw[line width=1pt, smooth, shift={(3 cm, 0 cm)}] plot[samples=200,domain=-2:-.5] (\x,{-(-\x-.5)*sqrt(4+2*\x)});

\end{tikzpicture}

\caption{The four types of symmetric networks with two triple junctions: the \textit{tree,} the \textit{lens,} the $\th$-\textit{network,} and the \textit{eyeglasses.}}
\label{fig:net-types}

\end{figure}

Before stating our main result, we mention that the case of the lens has already been studied and settled in a slightly more general case in \cites{Bellettini-Novaga-Lens2011, Schnuerer-etal-Lens2011}.
In complete analogy with \cites{GageHamilton-Convex1986, Grayson-RoundPoints1987}, it is proved that a (non-compact) lens shaped network which is symmetric with respect to one axis eventually becomes convex and approaches a straight line in finite time, as the enclosed region desappears and the curvature blows-up.
The following theorem gives a complete description in the remaining cases.

\begin{thm}\label{thm:main}
Let $\G_0$ be a symmetric regular network with two triple junctions.
Then there is a maximal time $T>0$ and a unique network flow $\{ \G(t) \}_{0 \leq t < T}$ with initial data $\G_0$, such that the set of its singular times is a finite subset of $(0,T]$, in particular there is no accumulation of singularities.
Moreover:
\begin{itemize}
    \item if $\G_0$ is a tree, then $T = \infty$ (global existence) and $\lim_{t \to \infty} \G(t)$ is either a (standard) cross or a Steiner tree;
    \item otherwise $T < \infty$, $\G(t)$ becomes eyeglasses-shaped after the last type-0 singularity, and the curvature blows up as the enclosed regions vanish with $t \uparrow T$.
\end{itemize}
\end{thm}

\begin{rem}
In the case of a tree, we cannot rule out a singularity at infinity, as the example in \cite{Pluda-Pozzetta-Lojasiewicz2023}*{Thm. 6.1} shows.
There the authors construct a globally defined flow which stays regular for every time and converges to a cross in infinite time.
\end{rem}

Let us briefly describe what are the dynamics in this situation.
Since the initial datum is symmetric, it is easy to see that the evolution also stays symmetric until the first singularity forms.
If we encounter a type-0 singularity, i.e. the length of one of the curves goes to zero and the curvature of the network remains bounded, then again by symmetry the vanishing curve must be the straight edge passing through the origin.
In this way, the two triple junctions collide, and we can the apply the results in \cite{Ilmanen-Neves-Schulze-Shortexistence-2019}*{Thm. 1.1} and \cite{LMPS-2021}*{Thm. 1.1} to restart the flow (cf. \cite{MNP-2juncs-2017}*{Thm. 6.1}).
Since there is a unique self-similar expanding soliton flowing out a standard cross, and it has the same symmetries as the cross \cite{Mazzeo-Saez-2011}*{Prp. 2.2}, we may conclude that the evolution remains symmetric as before.
A tree transitions to a tree, and a $\th$-network transitions to eyeglasses and vice-versa.
This process can continue as long as the curvature remains bounded, and as stated, this can only happen a finite number of times, so any oscillatory behaviour is excluded.

The proof of Theorem \ref{thm:main} relies on a result by Angenent \cite{AngParabolicII91}*{Thm 1.3}, which we present as Proposition \ref{prp:intersecs}, adapted to this singular case.
Its proof is grounded in the \textit{Sturmian theorem}, as stated by Angenent \cite{AngSing91}*{Thm. 2.1} (cf. \cite{AngZero88}*{Thms. C and D}).
In essence, this theorem asserts that if $u \in C^\infty(Q_T)$ is a solution to a linear parabolic equation in $Q_T := [0,1] \times [0,T]$, and if $u(x,t) \neq 0$ for all $0 \leq t \leq T$ and $x = 0,1$, then at any time $t \in (0,T]$, the number of zeroes of $u(\cdot,t)$ will be finite.
Furthermore, this number decreases as a function of $t$ and strictly decreases whenever $u(\cdot,t)$ has a multiple zero.
For additional applications, we refer to \cite{AngNodal91}, and for a detailed proof, the reader may consult \cite{AngZero88}.

Although we are dealing with a very special case of the network flow, it is reasonable to believe that Theorem \ref{thm:main} holds true in general for tree-like networks, with appropriate modifications.

\begin{conj}
The number of singular times during the evolution of a tree is finite.
If no boundary curve disappears during the evolution, then the flow exists for all positive times and converges to a (possibly degenerate) minimal network.
\end{conj}

The plan of the paper is the following: in Section \ref{sec:setup} we introduce the notion of network and network flow, and we recall some preliminary results on existence and uniqueness of solutions. 
In Section \ref{sec:proof} we prove our main result on the singularities of the flow of symmetric networks with two triple junctions.
Finally, in Section \ref{sec:extension} we extend the result to the flow of symmetric networks on the 2-sphere.


\section{Notation and preliminary results}\label{sec:setup}

Before proceeding to the proof of Theorem \ref{thm:main}, we shall first establish our notation, present the necessary definitions, and recapitulate the essential results.

Let $\g: [0,1] \to \R^2$ be a regular $C^2$ curve, meaning that $\g'(x) \neq 0$ for all $x \in [0,1]$. We denote its unit tangent as $\t(x) := \g'(x)/|\g'(x)|$ and its unit normal as $\n$, such that $\{ \t, \n \}$ is a positive basis of $\R^2$.
The curvature of $\g$ with respect to $\n$ is denoted as $\k$.
Occasionally, we may use superscript indices to label curves, and when we do, we will also label their tangents, normals, and curvatures accordingly.

\begin{defn}[Network] \label{def:net}
A \textit{network} $\G$ is a finite union of embedded, regular curves $\{ \g^j \}_{j=1}^n$ of class $C^2$, called \textit{edges}, that meet only at their endpoints and nontangentially, and such that the union of their images $\bigcup_{j=1}^n \g^j([0,1])$ is a connected set.
A network $\G$ is said to be \textit{regular} when its edges intersect solely at triple junctions, at which their interior tangents form equal angles of $2\pi/3$.
The endpoints of curves that are not shared by other curves are referred to as \textit{endpoints} of the network.
\end{defn}

Note that the regularity condition at triple junctions can be stated as follows: if three curves $\g^{j_k}$ ($k = 1,2,3$) intersect at, say, $x=0$, then
\[
\t^{j_1}(0) + \t^{j_2}(0) + \t^{j_3}(0) = 0 .
\]

\begin{defn}[Network flow]\label{def:netflow}
Let $\G(t) = \{ \g^j(\cdot,t) \}_{j=1}^n$, with $t \in (a,b)$, be one-parameter family of regular networks, with fixed endpoints $p^1, \ldots, p^r$, and time-dependent triple junctions $o^1(t), \ldots, o^s(t)$.
Then $\{ \G(t) \}_{a < t < b}$ is said to be a solution to the \textit{network flow} if at every time $t \in (a,b)$, with possible curve relabelling, the following system is satisfied:
\begin{equation}\label{eq:netflow}
\left\{ \begin{alignedat}{3}
\langle \p_t\g^j(x,t) , \n^j(x,t) \rangle & = \k^j(x,t) , &\quad &x \in [0,1] , \quad j = 1, \ldots, n , \\
\g^k(1,t) & = p^k , &\quad &k = 1, \ldots, r ,\\
\t^{l_1} + \t^{l_2} + \t^{l_3} & = 0 &\quad &\text{at the triple junction } o^l(t) , \quad l = 1, \ldots, s .
\end{alignedat}\right.
\end{equation}
\end{defn}

We now state a version of short time existence for the network flow that suits our needs.

\begin{prp}[Cf. \cite{Goesswein-Menzel-Pluda-Existence2023}*{Thms. 1.1--2}]\label{prp:reg-short-time}
Let $\G_0$ be a regular network.
Then there exists a smooth solution $\{ \G(t) \}_{0 \leq t < T}$, unique up to reparametrisations, to the network flow \eqref{eq:netflow} starting at $\G_0$ and fixing its endpoints.
Moreover, the flow can be extended to a maximal time $T>0$, at which at least one of the following scenarios unfolds:
\begin{itemize}
    \item $T = \infty$;
    \item the inferior limit as $t \uparrow T$ of the length of one of the edges of $\G(t)$ is zero;
    \item the superior limit as $t \uparrow T$ of the $L^2$-norm of the curvature of $\G(t)$ is infinite.
\end{itemize}
\end{prp}

Proposition \ref{prp:reg-short-time} allows us to initiate the flow from any regular data and extend it to a maximal time while ensuring the flow remains smooth.
Nevertheless, in the specific case we address here, where the initial network possesses only two triple junctions, we can provide additional insights.

\begin{prp}[Cf. \cite{MNP-2juncs-2017}*{Thm. 1.1}]\label{prp:2juncs}
Let $\G_0$ be a regular network with exactly two triple junctions, and let $\{ \G(t) \}_{0 \leq t < T}$ be the maximal smooth flow starting at $\G_0$.
Suppose also that $T$ is finite and the length of no boundary edge goes to zero.
Then, as $t \uparrow T$, one of the following occurs:
\begin{itemize}
    \item the length of a curve joining the triple junctions goes to zero while the curvature remains bounded;
    \item the limit of the lengths of the curves composing a loop goes to zero and the $L^2$-norm of the curvature goes to infinity.
\end{itemize}
If the network $\G_0$ is a tree, then only the first situation happens.
\end{prp}

As explained in the Introduction, when the first scenario in Proposition \ref{prp:2juncs} occurs, the flow approaches an irregular network with a single quadruple junction, resulting in the development of a type-0 singularity.
Nonetheless, a transition to a regular flow is made possible by the following proposition.

\begin{prp}[Cf. \cite{LMPS-2021}*{Thm. 1.1, Prp. 8.5}]\label{prp:irreg-short-time}
Let $\G_0$ be an irregular network.
Then there exists a solution $\{ \G(t) \}_{0 \leq t < T}$ to the network flow \eqref{eq:netflow} such that $\G(t)$ converges in the Hausdorff distance to $\G_0$ as $t \downarrow 0$.
Furthermore, all the solutions, accounting for possible reparametrizations, can be classified by the self-similar, tree-like expanding solitons described in \cite{Mazzeo-Saez-2011} at each irregular junction.
In particular, if $\G_0$ consists of a single quadruple junction with angles $\pi/3$ and $2\pi/3$, and no other junctions are present, then the flow admits a unique solution.
\end{prp}

While the convergence to the initial datum $\G_0$ can be understood in a much stronger sense, as discussed in \cite{LMPS-2021}, for our purposes, local uniform convergence suffices.
In any case, it is worth noting that convergence remains smooth away from the irregular junctions, provided that $\G_0$ itself is smooth.

Proposition \ref{prp:irreg-short-time} can be regarded as a restarting theorem after the formation of an irregular network, as previously explained.
Furthermore, since the flow remains regular for positive times, we can employ Proposition \ref{prp:reg-short-time} to extend it until the next singularity, if there is one.
This leads us to the following definition, which is the primary focus of our discussion here.

\begin{defn}[Extended Network Flow]\label{def:extnetflow}
An \textit{extended network flow} with initial condition $\G_0$ is a one-parameter family of networks $\{ \G(t) \}_{0 \leq t < T}$ that satisfies the following conditions:
\begin{itemize}
    \item there exists a finite number of times $0 = t_0 < t_1 < \cdots < t_m = T$, such that the restriction $\{ \G(t) \}_{t_{k} < t <t_{k+1}}$, $k = 0, \ldots, m-1$, is a regular network flow in the sense of Definition \ref{def:netflow};
    \item at each $t_k$, $k = 1, \ldots, m-1$, a type-0 singularity forms, and we call these \textit{singular times};
    \item $\G(t)$ converges to $\G(t_k)$ in the Hausdorff distance as $t \downarrow t_k$.
\end{itemize}
\end{defn}

Thus, another way then to rephrase Theorem \ref{thm:main} is that every solution to the network flow in the sense of Definition \ref{def:extnetflow} can be continued to a maximal extended flow such that either $T = \infty$, or else $T < \infty$ and the curvature increases without bound as $t \uparrow T$.

We now present the main tool used for the proof of Theorem \ref{thm:main}.

\begin{prp}[Cf. \cite{AngParabolicII91}*{Thm. 1.3}]\label{prp:intersecs}
Let $\g^1, \g^2: [0,1] \times [0,T) \to \R^2$ be two solutions to the curve shortening flow, and suppose that for each $(x,t) \in [0,1] \times [0,T)$
\[
\g^1(0,t) , \g^1(1,t) \neq \g^2(x,t) \quad \text{and} \quad \g^2(0,t) , \g^2(1,t) \neq \g^1(x,t) .
\]
Then the number of intersections of $\g^1(\cdot,t)$ and $\g^2(\cdot,t)$,
\[
i(t) := \#\{ (x_1,x_2) \in [0,1] \times [0,1] : \g^1(x_1,t) = \g^2(x_2,t) \} ,
\]
is finite for every $t \in (0,T)$.
Moreover, $i(\cdot)$ is a nonincreasing function of $t$, and decreases exactly when $\g^1$ and $\g^2$ become tangent at some point.
\end{prp}

We briefly remark that the presence of triple junctions is what makes it difficult to apply Proposition \ref{prp:intersecs} to a general, non-symmetric network.


\section{Proof of Theorem \ref{thm:main}}\label{sec:proof}

We are now ready for the proof of our main result.
Let $\G_0$ be a symmetric regular network with two triple junctions, and let $\{ \G(t) \}_{0 \leq t < T}$ be the evolution given by Proposition \ref{prp:reg-short-time}.
By our symmetry assumptions, from the system \eqref{eq:netflow} it follows that at each time $t \in (0,T)$ the network $\G(t)$ is also symmetric with the same axes of symmetry as $\G_0$.
Therefore, modulo a rotation and translation, the flow is completely described by one single curve in the first quadrant of $\R^2$, where the coordinate axes coincide with the symmetry axes of $\G(t)$.
Let $\g: [0,1] \times [0,T) \to \R^2$ be the evolution of this defining curve, and call $(x_1,x_2)$ the rectangular coordinates of $\R^2$.
After a reparametrisation, we may further suppose that $\g(0,t)$ is the triple junction, which lies in the $x_1$-axis, and $\t(0,t) = ( \frac{1}{2}, \frac{\sqrt{3}}{2} )$ is the unit tangent, which is constant.
Thus, we have the following boundary conditions for the curve evolution of $\g$, according to each case, if $\G_0$:
\begin{itemize}
    \item is a tree, then $\g(1,t) = p$ is a fixed point;
    \item is a $\th$-network, then $\g(1,t)$ is a free point in the $x_2$-axis such that $\t(1,t) = (-1,0)$;
    \item is eyeglasses, then $\g(1,t)$ is a free point in the $x_1$-axis such that $\t(1,t) = (0,-1)$.
\end{itemize}
As a consequence of Proposition \ref{prp:2juncs}, the flow develops a type-0 singularity at $T$ if and only if $\lim_{t \uparrow T} \g(0,t) = (0,0)$.

\begin{proof}[Proof of Theorem \ref{thm:main}]
To fix ideas, let us suppose that the initial network $\G_0$ is a tree with fixed endpoints.
Suppose also that at a finite time $T$ the flow develops a type-0 singularity.
We can then apply Proposition \ref{prp:irreg-short-time} and extend the flow a little further to a time $\widehat{T} > T$.
By symmetry, this can be viewed as extending the evolution of the defining curve to a map $\g: [0,1] \times [0, \widehat{T}) \to \R^2$, where the triple junction $\g(0,t)$ is in the $x_2$-axis for times $t \in (T,\widehat{T})$.

Now, consider a straight line $\ell$ through the origin, such that the endpoint $p \notin \ell$ and the angle between the $x_1$-axis and $\ell$ is in the range $( \frac{\pi}{6} , \frac{\pi}{3} )$.
Note also that $\ell$ is a static solution to the curve shortening flow.
Define the function
\[
i(t) := \# \{ x \in [0,1]: \g(x,t) \in \ell \} , \quad t \in (0, \widehat{T}).
\]
We will show that $i$ is finite and nonincreasing in time, and decreases strictly across $t = T$.

Indeed, as long as the vertices of the network do not collide, i.e. as $\g(0,t)$ stays away from the origin, we can invoke Proposition \ref{prp:intersecs} to see that $i(t)$ is not increasing in time.
On the other hand, $\g(\cdot,t)$ converges smoothly to a curve $\g(\cdot,T)$ as $t \uparrow T$, such that $\g(0,T) = 0$, and $\t(0,T) =  ( \frac{1}{2} , \frac{\sqrt{3}}{2} )$.
Therefore, there exist some small $\e, \d > 0$ such that $\g(\cdot,t)$ crosses $\ell$ in the rectangle $[0,\e] \times [0,2\e]$ at a single point, for every $t \in (T - \d, T]$.
Note that $\g(\cdot,T)$ intersects $\ell$ exactly at the origin, and besides this point it lies completely above $\ell$.
We will now show that, after we restart the flow, $\g(\cdot,t)$ remains at a positive distance above $\ell$ in $[0,\e] \times [0,2\e]$ for a short time.

Because of the invariance under reparametrisations, we can locally represent the evolution of $\g$ as a graph $u(x,t)$ over the $x_1$-axis, for $(x,t) \in [0,\e] \times [T, T+\d)$, with $\d$ possibly smaller.
The function $u: [0,\e] \times [T,T+\d) \to \R$ then satisfies the partial differential equation
\[
u_t = \frac{u_{xx}}{1+(u_x)^2} \quad \text{in} \quad [0,\e] \times (T,T+\d)
\]
with Cauchy-Neumann boundary conditions
\[
u(x,T) = u_T(x) , \quad u_x(0,t) = 1/\sqrt{3} , \quad (x,t) \in [0,\e] \times (T, T+\d) ,
\]
where $u_T: [0,\e] \to \R$ is a function parametrising $\g(\cdot,T)$.
If we consider the function $w(x,t) := u(x,t) - mx$, with $m \in ( \frac{1}{\sqrt{3}} , \sqrt{3} )$ being the tangent of the angle $\ell$ forms with the $x_1$-axis, then $w$ solves
\[
w_t = \frac{w_{xx}}{1+(w_x+m)^2} ,
\]
which is strictly parabolic.
Thanks to the estimate on the shortest curve of the flow with singular initial data \cite{Ilmanen-Neves-Schulze-Shortexistence-2019}*{Thm. 1.1}, there is a positive constant $c$ such that
\[
w(0,t) \geq c\sqrt{t-T} .
\]
Furthermore, Proposition \ref{prp:irreg-short-time} implies that as $t$ approaches $T$ from above, $u(\cdot, t)$ uniformly converges to $u_T$.
Therefore, for sufficiently small $\d$, we have $w(\e, t) > 0$ for all $t \in [T, T+\d)$.
This, combined with the fact that $w(x,T) \geq 0$ for all $x \in [0, \e]$ and the application of the maximum principle, shows that $w$ remains greater than zero in $[0, \e] \times [T, T+\d)$, which means $\g(\cdot, t)$ remains above the line $\ell$ during this interval.
Hence, in a neighbourhood of the origin, the number of intersections between $\g(\cdot, t)$ and $\ell$ decreases by precisely one as $t$ crosses $T$.
Outside this neighbourhood, we can once again employ Proposition \ref{prp:intersecs} for $t$ in the range $(T, T+\d)$.
This analysis demonstrates that $i(t)$ decreases by at least one over the interval $(0, T+\d)$.
For illustration purposes, see Figure \ref{fig:i-decrease-sing}.

\begin{figure}
\centering

\begin{tikzpicture}[x=1cm,y=1cm,scale=1.4]

\draw[line width=.5pt, smooth, shift={(-4 cm, 0 cm)}] (0,0) -- (2,0);
\draw[line width=.5pt, smooth, shift={(-4 cm, 0 cm)}] (0,0) -- (0,2);
\draw[line width=.5pt, dashed, shift={(-4 cm, 0 cm)}] (0,0) -- (2,2);
\draw[line width=1pt, smooth, shift={(-4 cm, 0 cm)}] plot[samples=200,domain=.5:2] (\x, {1.73*(\x-.5)*(\x+.5)-1.5*(\x-.5)^3});

\draw[line width=.5pt, smooth] (0,0) -- (2,0);
\draw[line width=.5pt, smooth] (0,0) -- (0,2);
\draw[line width=.5pt, dashed] (0,0) -- (2,2);
\draw[line width=1pt, smooth] plot[samples=200,domain=0:2] (\x,{-.5*\x*(\x-3.46)});

\draw[line width=.5pt, smooth, shift={(4 cm, 0 cm)}] (0,0) -- (2,0);
\draw[line width=.5pt, smooth, shift={(4 cm, 0 cm)}] (0,0) -- (0,2);
\draw[line width=.5pt, dashed, shift={(4 cm, 0 cm)}] (0,0) -- (2,2);
\draw[line width=1pt, smooth, shift={(4 cm, 0 cm)}] plot[samples=200,domain=0:2] (\x,{.5+.58*\x*(1+.04*(\x^2-\x^3))});

\end{tikzpicture}

\caption{The number of intersections with the diagonal (dashed line) decreases by at least one through a standard transition.}\label{fig:i-decrease-sing}

\end{figure}
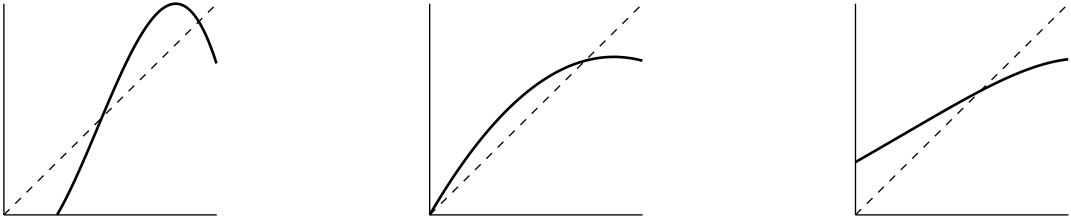

Due to our symmetry assumptions and our choice of the line $\ell$, after reflecting with respect to the diagonal $x_1 = x_2$ we find ourselves back in the initial setup.
As a consequence, we can repeat the previous reasoning every time the two triple junctions of $\G(t)$ coalesce into the origin, and as $i(t)$ cannot decrease indefinitely, it must become constant for sufficiently large $t>0$, after which there are no more type-0 singularities.
We can thus obtain an extended network flow as in Definition \ref{def:extnetflow}.

The steps described above carry almost identically to the other types of networks, with the only difference that we do not need to be concerned about avoiding any particular point $p$ because there are no external endpoints.

We conclude the proof by referencing once again Proposition \ref{prp:2juncs}, from which it follows that the flow of a tree can be extended indefinitely.
In contrast, for the $\th$-network and the eyeglasses cases, there must exist a time at which the two bounded regions collapse simultaneously, causing the $L^2$-norm of the curvature to blow up.
This can only happen as an eyeglasses-shaped network, as there is no self-similar shrinking $\th$-network \cite{BaldiHaussMantegazza-NoShrinkingTh2018}.

Finally, in the case of a tree, there exists a sequence $t_n \to \infty$ such that $\G(t_n)$ converges in $C^{1,\a} \cap W^{2,2}$ for every $\a \in (0,\frac{1}{2})$ to either a regular Steiner tree or a standard cross. If it converges to a Steiner tree, we can apply \cite{Pluda-Pozzetta-Lojasiewicz2023}*{Thm. 1.2} to establish the full smooth convergence of the flow as $t \to \infty$.
Otherwise, regardless of the sequence of times, the limit is a standard cross.
Thus, in this scenario as well, we observe full and smooth convergence.
\end{proof}


\section{Extension to the 2-sphere}\label{sec:extension}

We conclude this note by extending Theorem \ref{thm:main} to the network flow on the sphere $\S^2$ instead of $\R^2$.
Thanks to the theory developed in \cites{Angenent-ParabolicI1990, AngParabolicII91}, this extension can be achieved almost effortlessly, albeit with a mild change in the flow's behaviour near the maximal time of existence.
Since the outcome remains unchanged when considering a tree-like initial configuration, we will focus on the case of the $\th$-network and eyeglasses.
First, we state the analogous theorem, and then we explain how to adapt the arguments presented in Section \ref{sec:proof} to obtain the results.

\begin{thm}
Let $\G_0$ be a symmetric, closed and regular network with two triple junctions on the sphere $\S^2$.
Then there is a maximal time $T>0$ and a unique network flow $\{ \G(t) \}_{0 \leq t < T}$ with initial data $\G_0$, such that the set of its singular times is a finite subset of $(0,T)$, in particular there is no accumulation of singularities.
Moreover, either:
\begin{itemize}
    \item $T = \infty$ and $\lim_{t \to \infty} \G(t)$ is a minimal $\th$-network;
    \item or $T < \infty$ and the curvature blows up as one of the enclosed regions vanish with $t \uparrow T$.
\end{itemize}
\end{thm}

In this context, “symmetric” means symmetry with respect to a reflection across two perpendicular great circles, which are the geodesics of $\S^2$.

It is worth noting that all the definitions presented in Section \ref{sec:setup} straightforwardly apply to this case, and all the propositions in that section remain valid.

Regarding the proof, the only change is that instead of counting intersections with a straight line, we use a great circle passing through the centre of symmetry of the network, making an angle greater than $\pi/6$ but less than $\pi/3$ with respect to a chosen great circle of symmetry.
The analogue of Proposition \ref{prp:intersecs} asserts that this number decreases during the evolution.
Across a type-0 singularity, we again represent the evolution locally as a graph over the chosen great circle using the exponential map.
The resulting equation for the evolution remains strictly parabolic, and the application of the maximum principle yields the desired strict monotonicity.
We therefore conclude that the flow can be extended until the curvature becomes unbounded as an enclosed region vanishes, or it can be extended indefinitely and converges to a minimal network.
In this case, the minimal network must be a $\th$-network, as minimal eyeglasses or 8-figures do not exist on the 2-sphere.
In particular, infinity is excluded as a singular time.




\end{document}